\documentclass[12pt,legno]{article}\usepackage[dvips]{graphicx}
\usepackage{amsmath,amssymb,latexsym,amsfonts,enumerate,theorem,longtable,amscd}
\newtheorem{thm}{Theorem}[section]
\newtheorem{cor}[thm]{Corollary}
\newtheorem{lem}[thm]{Lemma}

\newtheorem{proof}{Proof}
\theoremstyle{definition}
\newtheorem{defn}[thm]{Definition}
\theoremstyle{remark}

\numberwithin{equation}{section}

\begin{document}

\begin{center}{\bf 
A Differential Calculus on the $(h,j)$-Deformed\\  Z$_3$-Graded Superplane}
\end{center}

\begin{center} {\bf Salih Celik and Sultan Celik}\\
Department of Mathematics, Yildiz Technical University, \\ DAVUTPASA-Esenler Istanbul, 34210 TURKEY
\end{center}

\begin{center} {\bf Erhan Cene}\\
Department of Statistics, Yildiz Technical University, \\ DAVUTPASA-Esenler Istanbul, 34210 TURKEY 
\end{center}


\date{October 22, 2013}

\begin{abstract}
In this work, Z$_3$-graded quantum $(h,j)$-superplane is introduced with a help of proper singular $g$ matrix and a Z$_3$-graded calculus is constructed over this new $h$-superplane. A new Z$_3$-graded $(h,j)$-deformed quantum (super)group is constructed via the obtained calculus.
\end{abstract}

\section{Introduction}

By the end of the twentieth century, quantum groups have started to draw great attention at the fields of mathematics and mathematical physics. Quantum groups was first defined in \cite{9}. After short time, quantum groups were generalized to quantum super groups which leads to an innovative mathematical field \cite{12}, applied this subject to Lie groups and Lie Algebras.

The noncommutative differential geometry of quantum groups was introduced by Woronowicz in \cite{19}. In this approach the quantum group was taken as the basic
noncommutative space and the differential calculus on the group was deduced from the properties of the group. The other approach, initiated by
Wess-Zumino \cite{18}, succeeded to extend Manin's emphasis \cite{15} on the quantum spaces as the primary objects, they defined differential forms in terms
of noncommuting (quantum) coordinates, and the differential and algebraic properties of quantum groups acting on these spaces are obtained from the properties
of the spaces. The natural extension of their scheme to superspace\cite{16} was introduced in \cite{3} and \cite{17}, for example.

Recently, there have been many attempts to generalize Z$_2$-graded constructions to the Z$_3$-graded case [\cite{1},\cite{4},\cite{7},\cite{10},\cite{11},\cite{13},\cite{14}]. Chung\cite{7} studied the Z$_3$-graded quantum space that generalizes the Z$_2$-graded space called a superspace, using the methods of \cite{18}. The first author of this paper investigated the noncommutative geometry of the Z$_3$-graded quantum superplane in \cite{4}. This work will follow the same pattern with one difference. In this work, differential geometry of $h$-deformed Z$_3$-graded quantum superplane is going to be investigated.

In $q$-differential calculus, exterior differential operator {\sf d} has two properties: nilpotency (that is, ${\sf d}^2=0$) and Leibniz Rule. In this work, it is assumed that ${\sf d}^2\neq{0}$ and ${\sf d}^3=0$, while constructing a calculus on Z$_3$-graded $h$-superplane, hence second order differentials are also considered in addition to the relations obtained in q-differential calculus. Thus, while $q$-commutation relations between differentials of coordinate functions and relations among differentials are given in $q$-differential calculus, additional relations will appear in Z$_3$-graded $h$-superplane, since second order differentials should be considered.

In this work, we shall build up the noncommutative differential calculus on the Z$_3$-graded $h$-superplane. The noncommutative differential calculus on the Z$_3$-graded $h$-superplane involves functions on the superplane, first and second differentials and differential forms.

The purpose of this paper is to present a differential calculus on the Z$_3$-graded $h$-superplane. The paper is organized as follows. In section 2 we obtain the Z$_3$-graded $h$-superplane via a contraction of Z$_3$-graded $q$-superplane using approach of \cite{2}. In section 3 we explicitly set up a differential calculus on the Z$_3$-graded $h$-superplane. Some relations are abtained in \cite{6}. In section 4 we find a {\bf new} Z$_3$-graded quantum supergroup denoted by GL$_{h,j}(1|1)$.

\section{The Algebra of Functions on the Z$_3$-graded $h$-Superplane}

It is well known that \cite{16} defined the Z$_2$-graded quantum superplane as an associative algebra whose even coordinate $x$ and the odd (Grassmann) coordinate $\theta$ satisfy
$$x\theta = q \theta x, \qquad \theta^2 = 0$$
where $q$ is a nonzero complex deformation parameter. One of the possible ways to generalize the quantum superplane is to use the power of nilpotency of
its odd generator. This fact gives the motivation for the following definition.

\begin{defn}
Let $K\{x',\theta'\}$ be a free algebra and $I_q$ is a two-sided ideal generated by $x'\theta'-q\theta'x'$ and $\theta'^3$. The Z$_3$-graded quantum superplane $K_q[x',\theta']$ is defined as quotient algebra $K\{x',\theta'\}/I_q$.
\end{defn}

Here, the coordinate $x'$ with respect to the Z$_3$-grading is of grade 0 and the coordinate $\theta'$ with respect to the Z$_3$-grading is of grade 1. Using the approach given in \cite{2}, $h$-deformation of Z$_3$-graded superplane will be described and afterwards a differential calculus on $h$-deformed structure will be constructed.

Recalling Definition 2.1, commutation relations between coordinate functions in Z$_3$-graded superplane can be given as follows
\begin{equation} \label{eq1}
x'\theta' = q\theta'x' \qquad \theta'^3=0.
\end{equation}
We consider a non-singular deformation matrix $g$ which is defined as in \cite{2},
\begin{equation} \label{eq2}
g = \left(\begin{matrix} 1 & 0 \\ \frac{h}{q-1} & 1 \end{matrix}\right)
\end{equation}
where $h$ is a new quantity having grade two. If we assume that,
\begin{equation} \label{eq3}
\left(\begin{matrix} x' \\ \theta' \end{matrix}\right) = g\left(\begin{matrix} x \\ \theta \end{matrix}\right).
\end{equation}
then, new coordinates $x$ and $\theta$ would be,
\begin{equation} \label{eq4}
x = x' \quad \mbox{and} \quad \theta = \theta' - \frac{h}{q-1} \, x'.
\end{equation}
If the relations (\ref{eq1}) is used in order to obtain commutation relation between $x$ and $\theta$ one can easily find the relation,
\begin{equation}\label{eq5}
x\theta = q\theta x+hx^2.
\end{equation}
While obtaining relation (\ref{eq5}) it is assumed that parameter $h$ is commutative with the coordinate $x$. Now let's assume that
\begin{equation} \label{eq6}
\theta h = qjh\theta \quad \mbox{and} \quad h^3=0
\end{equation}
where $j=e^\frac{2\pi i}{3}$  $(i^2 = - 1)$ and
$$j^3 = 1 \quad \mbox{and} \quad j^2 + j +1 = 0, \quad \mbox{or} \quad (j + 1)^2 = j.$$
If the coordinate $\theta'$ in (\ref{eq4}) is substituted in the second equation in (\ref{eq1}), then it can be found that,
\begin{equation} \label{eq7}
\theta^3=0.
\end{equation}
Consequently, in the limit $q\to1$, the relations that define Z$_3$-graded $h$-superplane can be obtained as defined in \cite{5}.
\begin{equation} \label{eq8}
x\theta=\theta x+hx^2, \quad  \theta^3=0, \quad h^3=0.
\end{equation}
Now we can define Z$_3$-graded $h$-superplane.
\begin{defn}
Let $K\{x,\theta,h\}$ be a free algebra and $I_h$ is a two-sided ideal generated by $x\theta-\theta x-hx^2$, $\theta^3$ and $h^3$. The  Z$_3$-graded $h$-superplane $K_h[x,\theta,h]$ is defined as quotient algebra $K\{x,\theta,h\}/I_h$.
\end{defn}

\section{A Differential Calculus on the Z$_3$-graded $h$-Superplane}

In this section, we construct a differential calculus on the Z$_3$-graded $h$-superplane. This calculus involves functions on the $h$-superplane, first and second differentials and differential forms. We begin with the definition of the Z$_3$-graded differential calculus. Let $\hat{\alpha}$ denotes the grade of $\alpha$.
\begin{defn}
Let $A$ be an arbitrary associative (in general, noncommutative) algebra and let $\Gamma^{\wedge n}$ be a space of $n$-form $(n=0,1,2)$ and $A$-bimodule.
A Z$_3$-graded differential calculus on the algebra $A$ is a Z$_3$-graded algebra $\Gamma^\wedge=\bigoplus_{n=0}^2 \Gamma^{\wedge n}$ with a ${\mathbb C}$
linear exterior differential operator ${\sf d}$ which defines the map ${\sf d}:\Gamma^\wedge \longrightarrow \Gamma^\wedge$ of grade one. A generalization of a usual differential calculus leads to the rules:
\begin{eqnarray} \label{eq9}
{\sf d}^3 & =& 0, \qquad ({\sf d}^2\ne0) \nonumber\\
{\sf d}(\alpha\wedge\beta) &=& ({\sf d}\alpha)\wedge\beta + j^{\hat{\alpha}} \, \alpha\wedge({\sf d}\beta), \nonumber\\
{\sf d}^2(\alpha\wedge\beta) &=& ({\sf d}^2\alpha)\wedge\beta + (j^{\hat{\alpha}}+j^{\hat{d\alpha}}) \, (d\alpha)\wedge({\sf d}\beta) + j^{2\hat{\alpha}} \, \alpha\wedge({\sf d}^2\beta)
\end{eqnarray}
for $\alpha\in\Gamma^{\wedge n}$ $(n=0,1,2)$ and $\beta\in\Gamma^{\wedge}$.
\end{defn}

\subsection{Some conventions and assumptions}

The Z$_3$-graded quantum superplane underlies a noncommutative differential calculus on a smooth manifold with exterior differential {\sf d} satisfying
${\sf d}^3=0$. So, in order to construct the differential calculus on the Z$_3$-graded quantum superplane a linear operator {\sf d} which acts on the functions of the coordinates of the Z$_3$-graded quantum superplane must be defined. For the definition, it is sufficient to define the action of {\sf d} on the coordinates and on their products:

The linear operator {\sf d} applied to $x$ produces a 1-form whose Z$_3$-grade is one, by definition. Similarly, application of {\sf d} to $\theta$ produces a 1-form whose Z$_3$-grade is two. We shall denote the obtained quantities by ${\sf d}x$ and ${\sf d}\theta$, respectively. When the linear operator {\sf d} applied to ${\sf d}x$ (or twice by iteration to $x$) it will produce a new entity which we shall call a 1-form of grade two, denoted by ${\sf d}^2x$ and to
${\sf d}\theta$ produces a 1-form of grade zero, modulo 3, denoted by ${\sf d}^2\theta$. Finally, we require that ${\sf d}^3=0$.

With a simple arithmetic calculation from (\ref{eq4}), one find
\begin{equation} \label{eq11}
x' = x \quad \mbox{and} \quad \theta' = \theta + \frac{h}{q-1} \, x.
\end{equation}

If the exterior differential operator {\sf d} is acted on both sides of the relations given with (\ref{eq11}) and by using the Leibniz Rule defined in (\ref{eq9}) will give
\begin{equation} \label{eq12}
{\sf d}x'  ={\sf d}x \quad \mbox{and} \quad {\sf d}\theta' = {\sf d}\theta + j \, \frac{h}{q-1} \, {\sf d}x.
\end{equation}
If {\sf d} is acted on (\ref{eq12}) once more, then we get,
\begin{equation} \label{eq13}
{\sf d}^2 x' = {\sf d}^2x \quad \mbox{and} \quad {\sf d}^2\theta' = {\sf d}^2\theta + j^2 \, \frac{h}{q-1} \, {\sf d}^2x.
\end{equation}
In order to obtain commutation relations between $h$ and differentials of coordinate functions, along with the assumption,
\begin{equation} \label{eq14}
x \,h=h \,x \quad \mbox{and} \quad \theta \,h = qjh \,\theta
\end{equation}
and also we made another assumption which is
\begin{equation} \label{eq15}
{\sf d} \, h = jh \, {\sf d}.
\end{equation}
If we apply the exterior differential operator {\sf d} to the relations in (\ref{eq14}) and use (\ref{eq15}), then we find,
\begin{equation} \label{eq16}
{\sf d}x \, h = jh \, {\sf d}x \quad \mbox{and} \quad {\sf d}\theta \, h = qj^2h \, {\sf d}\theta.
\end{equation}
Applying {\sf d} to the relations in (\ref{eq16}), will give
\begin{equation} \label{eq17}
{\sf d}^2x \, h = j^2h \, {\sf d}^2x \quad \mbox{and} \quad {\sf d}^2\theta \, h = qh \, {\sf d}^2\theta.
\end{equation}
Equation (\ref{eq14})-(\ref{eq17}) will be used in the proceeding sections where commutation relations are needed between $x, \theta, {\sf d}x, {\sf d}\theta,{\sf d}^2x,{\sf d}^2\theta$ and ${\sf d}$.

\subsection{Relations between coordinate functions and their first order differentials}

In this subsection, possible relations between the coordinate functions of Z$_3$-graded $h$-superplane and their differentials will be obtained by the
help of relations given with (\ref{eq18}) in below.

We assume that the commutation relations between the coordinates of $q$-superplane and their differentials are in the following form:
\begin{eqnarray} \label{eq18}
x' \,{\sf d}x' &=& A \, {\sf d}x' \,x', \nonumber\\
x' \,{\sf d}\theta' &=& F_{11} \, {\sf d}\theta' \,x' + F_{12} \, {\sf d}x' \,\theta', \nonumber\\
\theta' \,{\sf d}x' &=& F_{21} \, {\sf d}x' \,\theta' + F_{22} \, {\sf d}\theta' \,x', \nonumber\\
\theta' \,{\sf d}\theta' &=& B \, {\sf d}\theta' \,\theta'. \end{eqnarray}

The coefficients $A$, $B$ and $F_{ik}$ are related the complex deformation parameters $q$ and $j$. In this work, we shall determine these coefficients finding
new relations on the Z$_3$-graded $h$-superplane.
\begin{thm}
$(q,j,h)$-deformed relations between the coordinate functions of Z$_3$-graded $h$-superplane and their differentials are in the form
\begin{eqnarray} \label{eq19}
x \,{\sf d}x &=& j^2 \, {\sf d}x \,x, \nonumber\\
x \,{\sf d}\theta &=& q \, {\sf d}\theta \, x + (j^2-1) \, {\sf d}x \, \theta + hj \, {\sf d}x \,x, \nonumber\\
\theta \,{\sf d}x &=& jq^{-1} \, {\sf d}x \,\theta - q^{-1}hj^2 \, {\sf d}x \,x, \nonumber\\
\theta \,{\sf d}\theta &=& j \, {\sf d}\theta \,\theta.
\end{eqnarray}
These relations will be rewritten at the limit $q\to1$ later.
\end{thm}

\begin{proof}
For completing the proof, relations given with (\ref{eq11}) and (\ref{eq12}) should be replaced with relations (\ref{eq18}) step by step. After some tedious calculations, relations (\ref{eq18}) would yield,
\begin{eqnarray}\label{eq20}
x \,{\sf d}x &=& A \, {\sf d}x \,x, \nonumber\\
x \,{\sf d}\theta &=& F_{11} \, {\sf d}\theta \,x + F_{12} \, {\sf d}x \,\theta + \frac{h}{q-1} \, (F_{11}j+F_{12}j-Aj) \, {\sf d}x \,x, \nonumber\\
\theta \,{\sf d}x &=& F_{21} \, {\sf d}x \,\theta + F_{22} \, {\sf d}\theta \,x + \frac{h}{q-1} \, (F_{21}j+F_{22}j-A) \, {\sf d}x \,x, \\
\theta \,{\sf d}\theta &=& B \, {\sf d}\theta \,\theta + \frac{h}{q-1} \, K_1 \, {\sf d}\theta \,x + \frac{h}{q-1} \, K_2 \, {\sf d}x \,\theta + \left(\frac{h}{q-1}\right)^2 \, K_3 \, {\sf d}x \,x \nonumber
\end{eqnarray}
where
\begin{eqnarray} \label{eq21}
K_{1} &=& Bj^2q-F_{22}j^2q-F_{11}, \nonumber\\
K_{2} &= & Bj-F_{21}j^2q-F_{12}, \nonumber\\
K_{3} &=& Bj^2-F_{21}q-F_{22}q+Aj^2q-F_{11}j-F_{12}j. \end{eqnarray}
So our problem is reduced to find the coefficients in relations (\ref{eq20}) and (\ref{eq21}). In order to do that, we will act {\sf d} to (\ref{eq5}) and (\ref{eq7}). Applying {\sf d} to (\ref{eq5}) will lead to
\begin{eqnarray} \label{eq22}
x \, {\sf d}\theta &=& (q+qjF_{22}){\sf d}\theta \, x+(qjF_{21}-1){\sf d}x \, \theta+\left[\frac{qjh}{q-1}(F_{21}j+F_{22}j-A)\right.\nonumber\\
&& +hj(A+1)\Big]{\sf d}x \, x.
\end{eqnarray}
Comparing equation (\ref{eq22}) with the second equation in (\ref{eq20}) would yield the equations
\begin{equation} \label{eq23}
F_{11}=q(1+jF_{22}) \quad \mbox{and} \quad F_{12}=qjF_{21}-1.
\end{equation}
If the exterior differential operator {\sf d} is acted on (\ref{eq7}), after some tedious calculations one can see that,
\[1+jB+j^2B^2=0.\]
Hence, it appears that $B=1$ or $B=j$. Since taking $B=1$ doesn't yield to a solution, we are going to take $B=j$. Other coefficients can be found by using $B=j$. Therefore, coefficients in (\ref{eq20}) are determined in terms of $q$ and $j$. Consequently, the relations given with (\ref{eq19}) is obtained.
\end{proof}

\subsection{Relations between coordinate functions and their second order differentials}

In this subsection, possible relations between the coordinate functions of Z$_3$-graded $h$-superplane and their second order differentials will be obtained.
\begin{lem}
Relation between ${\sf d}x$ and ${\sf d}\theta$ is
\begin{equation} \label{eq24}
{\sf d}x\wedge{\sf d}\theta = F \, {\sf d}\theta\wedge{\sf d}x+\frac{h}{q-1}(Fj-j^2)({\sf d}x\wedge{\sf d}x)
\end{equation}
where $F$ depends on $q$ and $j$.
\end{lem}

\begin{proof} In Z$_3$-graded $q$-superplane this relation is at the form of
\begin{equation*}
{\sf d}x'\wedge{\sf d}\theta' = F \, {\sf d}\theta'\wedge{\sf d}x'.
\end{equation*}
By using (\ref{eq12}), at the left side,
\begin{equation*}
{\sf d}x'\wedge{\sf d}\theta'={\sf d}x\wedge\left({\sf d}\theta+j\frac{h}{q-1}{\sf d}x\right)={\sf d}x\wedge{\sf d}\theta+j^2\frac{h}{q-1}({\sf d}x\wedge{\sf d}x).
\end{equation*}
and right side
\begin{equation*}
F \, {\sf d}\theta'\wedge{\sf d}x' = F \, \left({\sf d}\theta + \frac{jh}{q-1} \, {\sf d}x\right)\wedge{\sf d}x = F \, {\sf d}\theta\wedge{\sf d}x + Fj \frac{h}{q-1} \, ({\sf d}x\wedge{\sf d}x).
\end{equation*}
Equality of these two equations would give the relation
\begin{equation*}
{\sf d}x\wedge{\sf d}\theta = F \, {\sf d}\theta\wedge{\sf d}x + \frac{h}{q-1} \, (Fj-j^2) \, ({\sf d}x\wedge{\sf d}x).
\end{equation*}
Here $F$ will be determined in Theorem 3.4.
\end{proof}

\begin{thm}
$(q,j,h)$-deformed relations between the coordinate functions of Z$_3$-graded $h$-superplane and their differentials are in the form
\begin{eqnarray} \label{eq25}
x \,{\sf d}^2x &=& j^2 \, {\sf d}^2x \,x,\nonumber\\
x \,{\sf d}^2\theta &=& q \, {\sf d}^2\theta \,x + (j^2-1) \, {\sf d}^2x \, \theta + hj^2 \, {\sf d}^2x \,x,\nonumber\\
\theta \,{\sf d}^2x &=& q^{-1} \, {\sf d}^2x \,\theta - q^{-1}hj^2 \, {\sf d}^2x \,x, \nonumber\\
\theta \,{\sf d}^2\theta &=& {\sf d}^2\theta \,\theta.
\end{eqnarray}
and differentials
\begin{equation} \label{eq26}
{\sf d}x\wedge{\sf d}\theta = qj \, {\sf d}\theta\wedge{\sf d}x + hj^2 \, ({\sf d}x\wedge{\sf d}x).
\end{equation}
\end{thm}

\begin{proof} Applying exterior differential operator {\sf d} to (\ref{eq19}) would give us the desired results. For the first equation in (\ref{eq19}), left side,
$${\sf d}\wedge(x{\sf d}x) = {\sf d}x\wedge{\sf d}x + x \,{\sf d}^2x,$$
and right side
$$j^2 \, {\sf d}\wedge({\sf d}x \,x) = j^2 \, d^2x \,x + j^2j \, ({\sf d}x\wedge{\sf d}x).$$
From the equality of two sides,
$$x \,{\sf d}^2x = j^2 \, d^2x \,x$$
can be obtained. Using similar approach and making necessary arrangements to the second equation in (\ref{eq19}) would yield,
\begin{eqnarray*}
x \,{\sf d}^2\theta &=& q \, {\sf d}^2\theta~x + (j^2-1) \, {\sf d}^2x \,\theta+j^2h \, {\sf d}^2x \,x + (-Fj+qj^2) \, {\sf d}\theta\wedge{\sf d}x\\
&&+\left[\frac{h}{q-1}(1-Fj^2)+h\right]({\sf d}x\wedge{\sf d}x).
\end{eqnarray*}
Having first order differentials in the relation between coordinate functions and second order differentials, violates homogeneity. In order to have a homogeneous relation, coefficients of ${\sf d}\theta\wedge{\sf d}x$ and ${\sf d}x\wedge{\sf d}x$ should be zero. Taking $F=qj$ would make those coefficient zero. Hence, desired equation would become,
$$x \,{\sf d}^2\theta = q \, {\sf d}^2\theta \,x + (j^2-1) \, {\sf d}^2x \,\theta + hj^2 \, {\sf d}^2x \,x.$$
Also, the relation (\ref{eq24}) given in Lemma 3.3 would transform to (\ref{eq26}) by taking $F=qj$.
One can find the third and fourth equations in (\ref{eq25}) by applying exterior differential operator ${\sf d}$ to third and fourth equation in (\ref{eq19}).
\end{proof}

\subsection{Relations between first order differentials and second order differentials}

In this subsection, relations between first order differentials and second order differentials of the coordinate functions will be obtained by using relations
in (\ref{eq25}).
\begin{lem}
$(q,j,h)$-deformed relations between first order differentials and second order differentials of the coordinate functions of Z$_3$-graded $h$-superplane and their differentials are in the form
\begin{eqnarray} \label{eq27}
{\sf d}x\wedge{\sf d}^2x &=& j \, {\sf d}^2x\wedge{\sf d}x, \nonumber\\
{\sf d}x\wedge{\sf d}^2\theta &=& q \, {\sf d}^2\theta\wedge{\sf d}x + (j-j^2) \, {\sf d}^2x\wedge{\sf d}\theta + hj^2 \, {\sf d}^2x\wedge{\sf d}x, \nonumber\\
{\sf d}\theta\wedge{\sf d}^2x &=& q^{-1} j^2 \, {\sf d}^2x\wedge{\sf d}\theta - q^{-1}hj^2 \, {\sf d}^2x\wedge{\sf d}x, \nonumber\\
{\sf d}\theta\wedge{\sf d}^2\theta &=& {\sf d}^2\theta\wedge{\sf d}\theta. \end{eqnarray}
\end{lem}

\begin{proof} For completing the proof, we are going to apply ${\sf d}$ exterior differential operator to the relations given with (\ref{eq25}). For the first equation, one can obtain left side,
\[{\sf d}\wedge(x \,{\sf d}^2x) = {\sf d}x\wedge{\sf d}^2x,\]
and right side
\[j^2 \, {\sf d}\wedge({\sf d}^2x \,x) = j \, {\sf d}^2x\wedge{\sf d}x.\]
From the equality of these equations, one can obtain
\[{\sf d}x\wedge{\sf d}^2x = j \, {\sf d}^2x\wedge{\sf d}x.\]
Other equations can be found by using same approach.
\end{proof}

\begin{cor}
The relationship between ${\sf d}^2x$ and ${\sf d}^2\theta$ as follows
\begin{equation} \label{eq28}
{\sf d}^2x\wedge{\sf d}^2\theta = qj^2 \, {\sf d}^2\theta\wedge{\sf d}^2x + jh \, {\sf d}^2x\wedge{\sf d}^2x.
\end{equation}
\end{cor}

\subsubsection{$Z_3$-graded $h$-superplane and some $(h,j)$-deformed relations}

In this subsection, we will obtain commutation relations on $Z_3$-graded $h$-superplane, by taking the limit $q\to1$ at the previously found relations.
\begin{itemize}
\item In equation (\ref{eq8}), relations between coordinate functions of $Z_3$-graded $h$-superplane was found as,
\begin{equation*}
x\theta = \theta x+hx^2, \quad  \theta^3 = 0, \quad h^3 = 0.
\end{equation*}
\end{itemize}

Following the same approach and taking $q\to1$ at the equations (\ref{eq19}),(\ref{eq25})-(\ref{eq28}) would give us the following relations.
\begin{itemize}
\item Relations between coordinate functions and their first order differentials
\begin{eqnarray*}
x \,{\sf d}x &=& j^2 \, {\sf d}x \,x,\\
x \,{\sf d}\theta &=& {\sf d}\theta \,x + (j^2-1) \, {\sf d}x \,\theta + hj \, {\sf d}x \,x,\\
\theta \,{\sf d}x &=& j \, {\sf d}x \,\theta - hj^2 \, {\sf d}x \,x, \\
\theta \,{\sf d}\theta &=& j \, {\sf d}\theta \,\theta.
\end{eqnarray*}
\item Relations between coordinate functions and their second order differentials
\begin{eqnarray*}
x \,{\sf d}^2x &=& j^2 \, {\sf d}^2x \,x,\\
x \,{\sf d}^2\theta &=& {\sf d}^2\theta \,x + (j^2-1) \, {\sf d}^2x \,\theta + hj^2 \, {\sf d}^2x \,x,\\
\theta \,{\sf d}^2x &=& {\sf d}^2x \,\theta - hj^2 \, {\sf d}^2x \,x, \\
\theta \,{\sf d}^2\theta &=& {\sf d}^2\theta \,\theta.
\end{eqnarray*}
\item Relations between first order differentials
\begin{eqnarray*}
{\sf d}x\wedge{\sf d}\theta = j \, {\sf d}\theta\wedge{\sf d}x + hj^2 \, ({\sf d}x\wedge{\sf d}x).
\end{eqnarray*}
\item Relations between first order differentials and second order differentials
\begin{eqnarray*}
{\sf d}x\wedge{\sf d}^2x &=& j \, {\sf d}^2x\wedge{\sf d}x,\\
{\sf d}x\wedge{\sf d}^2\theta &=& {\sf d}^2\theta\wedge{\sf d}x + (j-j^2) \, {\sf d}^2x\wedge{\sf d}\theta + hj^2 \, {\sf d}^2x\wedge{\sf d}x,\\
{\sf d}\theta\wedge{\sf d}^2x &=& j^2 \, {\sf d}^2x\wedge{\sf d}\theta - hj^2 \, {\sf d}^2x\wedge{\sf d}x, \\
{\sf d}\theta\wedge{\sf d}^2\theta &=& {\sf d}^2\theta\wedge{\sf d}\theta. \end{eqnarray*}
\item Relations between second order differentials
\[{\sf d}^2x\wedge{\sf d}^2\theta = j^2 \, {\sf d}^2\theta\wedge{\sf d}^2x + jh \, {\sf d}^2x\wedge{\sf d}^2x.\]
\end{itemize}

\subsection{The Relations Between Partial Derivatives and First and Second Order Differentials}

In this section, we are going to obtain relations between the coordinate functions and their partial derivatives and also relations between first order differentials and their partial derivatives on the $Z_3$-graded $h$-superplane.
\begin{defn} If $f$ is a differentiable function of $x$ and $\theta$, then first order differential of $f$ is defined as
\begin{equation} \label{eq29}
{\sf d}f=({\sf d}x\partial_x+{\sf d}\theta\partial_\theta)f.
\end{equation}
\end{defn}

\subsubsection{Relations between the coordinate functions and partial derivatives}
In this subsection, commutation relations between the coordinate functions and partial derivatives will be given.
\begin{thm}
Commutation relations between the coordinate functions and partial derivatives are given with
\begin{eqnarray} \label{eq30}
\partial_xx &=& 1 + j^2x\partial_x + (j^2-1)\theta\partial_\theta+hx\partial_\theta, \nonumber\\
\partial_\theta x &=& qx\partial_\theta, \nonumber\\
\partial_x\theta &=& j^2q^{-1}(\theta-hx)\partial_x, \nonumber\\
\partial_\theta \theta &=& 1 + j^2\theta\partial_\theta.
\end{eqnarray}
\end{thm}

\begin{proof}
Writing $xf$ instead of $f$ in (\ref{eq29}) would give the left side,
\begin{eqnarray*}
{\sf d}(xf) &=& {\sf d}xf+x{\sf d}f = {\sf d}xf+x({\sf d}x\partial_x + {\sf d}\theta\partial\theta)f \\
&=& \left[ {\sf d}x(1 + j^2x\partial_x + (j^2-1)\theta \partial_\theta + hx\partial_\theta){\sf d}\theta\partial_\theta x\right] f
\end{eqnarray*}
and right side
\begin{equation*}
{\sf d}(xf) = ({\sf d}x\partial_x x + {\sf d}\theta\partial_\theta x)f
\end{equation*}
from the equality of those two relations, desired equations can be obtained.
\end{proof}

\subsubsection{Relations between partial derivatives and first order differentials}

In this subsection, commutation relations between the first order differentials and partial derivatives will be given.
\begin{thm}
Commutation relations between the first order differentials and partial derivatives are given with:
\begin{eqnarray} \label{eq31}
\partial_x{\sf d}x &=& j{\sf d}x\partial_x - j^2h{\sf d}x\partial_\theta, \nonumber\\
\partial_x{\sf d}\theta &=& q^{-1}{\sf d}\theta\partial_x + q^{-1}jh{\sf d}x\partial_x, \nonumber\\
\partial_\theta{\sf d}x &=& qj^2{\sf d}x\partial_\theta, \nonumber\\
\partial_\theta{\sf d}\theta &=& (j^2-j){\sf d}x\partial_x + j^2{\sf d}\theta\partial_\theta.
\end{eqnarray}
\end{thm}

\begin{proof} First let's assume that these relations are at the form of
\begin{eqnarray} \label{eq32}
\partial_x{\sf d}x &=& A_1{\sf d}x\partial_x + A_2{\sf d}\theta\partial_\theta + A_3{\sf d}x\partial_\theta + A_4{\sf d}\theta\partial_x, \nonumber\\
\partial_x{\sf d}\theta &=& A_5{\sf d}\theta\partial_x + A_6{\sf d}x\partial_\theta + A_7{\sf d}x\partial_x + A_8{\sf d}\theta\partial_\theta, \nonumber\\
\partial_\theta{\sf d}x &=& A_9{\sf d}x\partial_\theta + A_{10}{\sf d}\theta\partial_x + A_{11}{\sf d}x\partial_x+A_{12}{\sf d}\theta\partial_\theta,\nonumber\\
\partial_\theta{\sf d}\theta &=& A_{13}{\sf d}x\partial_x + A_{14}{\sf d}\theta\partial_\theta + A_{15}{\sf d}x\partial_\theta + A_{16}{\sf d}\theta\partial_x.
\end{eqnarray}
From the definition of partial derivative operator, we know that:
\begin{equation} \label{eq33}
\partial_i(x^i{\sf d}x^k) = \delta_j^i\delta_l^k{\sf d}x^k, \quad (x^1 = x, \quad x^2 = \theta).
\end{equation}
Acting partial derivative operator to the first equation in (\ref{eq19}) would yield, $\partial_x(x~{\sf d}x-j^2 {\sf d}x~x)=0$. Using (\ref{eq33}) would give us,
${\sf d}x-j^2\partial_x{\sf d}x~x=0$. If this equation is written at the proper place in (\ref{eq32}), one can find
\begin{eqnarray*}
{\sf d}x - j^2\left[A_1{\sf d}x\partial_x+A_2{\sf d}\theta\partial_\theta+A_3{\sf d}x\partial_\theta+A_4{\sf d}\theta\partial_x\right] x &=& 0 \nonumber\\
{\sf d}x - j^2A_1{\sf d}x - j^2A_4{\sf d}\theta &=& 0 \nonumber\\
(1 - j^2A_1){\sf d}x - j^2A_4{\sf d}\theta &=& 0.
\end{eqnarray*}
From here it can be easily seen that $A_1=j$ and $A_4=0$. All $A_i$ coefficients can be obtained after some messy calculations by acting both $\partial_x$ and $\partial_\theta$ to the all equations in (\ref{eq19}).
\end{proof}

\subsubsection{Relations between partial derivatives}

In this subsection, commutation relations between partial derivatives will be given.
\begin{thm}
Relations between the partial derivatives are
\begin{eqnarray} \label{eq34}
\partial_x\partial_\theta &=& jq\partial_\theta\partial_x, \nonumber\\
\partial_\theta^3 &=& 0. \end{eqnarray}
\end{thm}

\begin{proof}
In $Z_3$-graded space we know that $d^3=0$. Hence,
\begin{eqnarray*}
{\sf d}^2f &=& \left[({\sf d}x\partial_x + {\sf d}\theta\partial_\theta)({\sf d}x\partial_x + {\sf d}\theta\partial_\theta)\right] f\\
&=& \left[{\sf d}x\partial_x{\sf d}x\partial_x + {\sf d}x\partial_x{\sf d}\theta\partial_\theta + {\sf d}\theta\partial_\theta{\sf d}x\partial_x +
    {\sf d}\theta\partial_\theta{\sf d}\theta\partial_\theta\right] f.\end{eqnarray*}
and
\begin{eqnarray*}
0 &=& d^3f \left[({\sf d}x\partial_x+{\sf d}\theta\partial_\theta)({\sf d}x\partial_x+{\sf d}\theta\partial_\theta)({\sf d}x\partial_x +
      {\sf d}\theta\partial_\theta)\right] f \nonumber\\
&=&\left[ ({\sf d}x\partial_x{\sf d}x\partial_x + {\sf d}x\partial_x{\sf d}\theta\partial_\theta + {\sf d}\theta\partial_\theta{\sf d}x\partial_x +
   {\sf d}\theta\partial_\theta{\sf d}\theta\partial_\theta)({\sf d}x\partial_x + {\sf d}\theta\partial_\theta)\right] f \nonumber\\
&=&[{\sf d}x\partial_x{\sf d}x\partial_x{\sf d}x\partial_x + {\sf d}x\partial_x{\sf d}x\partial_x{\sf d}\theta\partial_\theta +
   {\sf d}x\partial_x{\sf d}\theta\partial_\theta{\sf d}x\partial_x + {\sf d}x\partial_x{\sf d}\theta\partial_\theta{\sf d}\theta\partial_\theta \nonumber\\
&& + {\sf d}\theta\partial_\theta{\sf d}x\partial_x{\sf d}x\partial_x + {\sf d}\theta\partial_\theta{\sf d}x\partial_x{\sf d}\theta\partial_\theta +
    {\sf d}\theta\partial_\theta{\sf d}\theta\partial_\theta{\sf d}x\partial_x + {\sf d}\theta\partial_\theta{\sf d}\theta\partial_\theta{\sf d}\theta\partial_\theta] f.
\end{eqnarray*}
Hence, using (\ref{eq17}) and (\ref{eq31}) in this equation, assuming ${\sf d}x\wedge{\sf d}x\wedge{\sf d}x=0$, and with the help of homogeneity one can obtain desired results.
\end{proof}

\subsubsection{Some $(h,j)$-deformed relations for partial derivatives}

Taking limit $q\to1$ in the equations (\ref{eq30}), (\ref{eq31}) and (\ref{eq34}) would give us new relations.
\begin{itemize}
\item Relations of the coordinate functions with partial derivatives and relations between partial derivatives
\begin{eqnarray} \label{eq35}
\partial_xx &=& 1 + j^2x\partial_x + (j^2-1)\theta\partial_\theta + hx\partial_\theta, \nonumber\\
\partial_\theta x &=& x\partial_\theta, \nonumber\\
\partial_x\theta &=& j^2(\theta-hx)\partial_x,\nonumber\\
\partial_\theta \theta &=& 1 + j^2\theta\partial_\theta, \nonumber\\
\partial_x\partial_\theta &=& j\partial_\theta\partial_x, \nonumber\\
\partial_\theta^3 &=& 0.
\end{eqnarray}

\item Relations between first order differentials of coordinate functions and partial derivatives
\begin{eqnarray*}
\partial_x{\sf d}x &=& j{\sf d}x\partial_x - j^2h{\sf d}x\partial_\theta,\\
\partial_x{\sf d}\theta &=& {\sf d}\theta\partial_x + jh{\sf d}x\partial_x,\\
\partial_\theta{\sf d}x &=& j^2{\sf d}x\partial_\theta,\\
\partial_\theta{\sf d}\theta &=& (j^2-j){\sf d}x\partial_x + j^2{\sf d}\theta\partial_\theta. \end{eqnarray*}
\end{itemize}

\begin{defn}
The Z$_3$-graded quantum Weyl algebra ${\mathcal A}_{h,j}(2)$ is the unital algebra with four generators $x$, $\theta$, $\partial_x$, $\partial_\theta$ and
defining relations (\ref{eq8}) and (\ref{eq35}).
\end{defn}

\subsection{Cartan-Maurer Forms}
In this section Cartan-Maurer forms will be described and necessary commutation relations will be obtained. In $Z_3$-graded $q$ deformation, two-forms had been described with the help of a generators of an $\mathcal{A}$ algebra in \cite{4}.
\begin{eqnarray*}
w' &=& {\sf d}x'(x')^{-1},\\
u' &=& {\sf d}\theta'(x')^{-1}-{\sf d}x'(x')^{-1}\theta'(x')^{-1}.
\end{eqnarray*}

{\bf Note:} Cartan-Maurer forms in Z$_3$-graded $h$-deformation are
\begin{eqnarray}\label{eq36}
w &=& {\sf d}x \, x^{-1}, \nonumber\\
u &=& {\sf d}\theta \, x^{-1}-{\sf d}x \, x^{-1}\theta \, x^{-1}\end{eqnarray}
under the assumptions of
\begin{equation} \label{eq37}
u \, h = qj^2h \, u, \quad w \, h = jh \, w
\end{equation}

\subsubsection{Relations between the coordinate functions and Cartan-Maurer forms}

In this subsection, relations between the coordinate functions and Cartan-Maurer forms will be obtained.
\begin{lem}
Relations between the coordinate functions and Cartan-Maurer forms are
\begin{eqnarray} \label{eq38}
xw &=& j^2 \, wx, \nonumber\\
xu &=& q \, ux, \nonumber\\
\theta w &=& j \, w\theta, \nonumber\\
\theta u &=& qj \, u\theta + qh \, ux.
\end{eqnarray}
\end{lem}

\begin{proof} Multiplying both sides of the first equation in (\ref{eq36}) with $x$ gives,
$$xw = x \,{\sf d}x \,x^{-1}.$$
Using convenient equation in relation system (\ref{eq19}) would give us,
$$xw = j^2 wx.$$
Other relations can be obtained, by using equations in (\ref{eq19}) and applying necessary transformations.
\end{proof}

\subsubsection{Relations between the Cartan-Maurer forms and first order differentials}

In this subsection, relations between the Cartan-Maurer forms and first order differentials will be given.
\begin{lem}
Relations between the Cartan-Maurer forms and first order differentials are,
\begin{eqnarray}\label{eq39}
w\wedge{\sf d}x &=& j \, {\sf d}x\wedge w, \nonumber\\
u\wedge{\sf d}x &=& q^{-1} \, {\sf d}x\wedge u, \nonumber\\
w\wedge{\sf d}\theta &=& j \, {\sf d}\theta\wedge w + (1-j) \, \theta x^{-1} {\sf d}x\wedge w, \nonumber\\
u\wedge{\sf d}\theta &=& q^{-1} \, {\sf d}\theta\wedge u + q^{-1}[(1-j) \, \theta x^{-1}-h] \,{\sf d}x\wedge u.
\end{eqnarray}
\end{lem}

\begin{proof} These relations can be found by using, (\ref{eq36}), (\ref{eq37}), (\ref{eq19}) and (\ref{eq25}), and making necessary arrangements.
\end{proof}

\subsubsection{Relations between the Cartan-Maurer forms and second order differentials}

In this subsection, relations between the Cartan-Maurer forms and second order differentials will be given.
\begin{lem}
Relations between the Cartan-Maurer forms and second order differentials are,
\begin{eqnarray}\label{eq40}
w\wedge{\sf d}^2x &=& j^2 \, {\sf d}^2x\wedge w, \nonumber\\
u\wedge{\sf d}^2x &=& q^{-1} \, {\sf d}^2x\wedge u, \nonumber\\
w\wedge{\sf d}^2\theta &=& (j-j^2)q^{-1} \, d^2x\wedge u + d^2\theta\wedge w, \nonumber\\
u\wedge{\sf d}^2\theta &=& q^{-1} \, {\sf d}^2\theta\wedge u + \left[(j-j^2)x^{-1}\theta - q^{-1}j^2h\right] {\sf d}^2x\wedge u.
\end{eqnarray}
\end{lem}
\begin{proof} These relations can be found by using, (\ref{eq36}), (\ref{eq37}), (\ref{eq19}), (\ref{eq25}) and (\ref{eq26}) and making necessary arrangements.
\end{proof}

\subsubsection{Relations between the Cartan-Maurer forms}

In this subsection, relations between the Cartan-Maurer forms will be given.
\begin{thm}
Relations between the Cartan-Maurer forms are,
\begin{eqnarray}\label{eq41}
w\wedge u&=& u\wedge w, \nonumber\\
w\wedge w\wedge w &=& 0.
\end{eqnarray}
\end{thm}
\begin{proof} These relations can be found by using, (\ref{eq19}) and (\ref{eq26}) and making necessary arrangements.
\end{proof}

\begin{cor}
The Cartan-Maurer forms are closed in the means that
\begin{equation}\label{eq42}
{\sf d}^2\wedge w = 0, \quad {\sf d}^2\wedge u =0.
\end{equation}
\end{cor}

\subsubsection{Some $(h,j)$-deformed relations for Cartan-Maurer forms}

Taking limit $q\to1$ in the equations (\ref{eq38}-\ref{eq40}) would give us new relations.
\begin{itemize}
\item Relations between the coordinate functions and Cartan-Maurer forms
\begin{eqnarray*}
xw &=& j^2wx, \\
xu &=& ux, \\
\theta w &=& j w\theta, \\
\theta u &=& j u\theta + h ux.
\end{eqnarray*}

\item Relations between the Cartan-Maurer forms and first order differentials
\begin{eqnarray*}
w\wedge{\sf d}x &=& j {\sf d}x\wedge w, \\
u\wedge{\sf d}x &=& {\sf d}x\wedge u, \\
w\wedge{\sf d}\theta &=& j{\sf d}\theta\wedge w + (1-j) \, \theta x^{-1} {\sf d}x\wedge w, \\
u\wedge{\sf d}\theta &=& {\sf d}\theta\wedge u - h \, {\sf d}x\wedge u + [(1-j) \, \theta x^{-1}-h] \,{\sf d}x\wedge u.
\end{eqnarray*}

\item Relations between the Cartan-Maurer forms and second order differentials
\begin{eqnarray*}
w\wedge{\sf d}^2x &=& j^2 \, {\sf d}^2x\wedge w, \\
u\wedge{\sf d}^2x &=& {\sf d}^2x\wedge u, \\
w\wedge{\sf d}^2\theta &=& (j-j^2) \, d^2x\wedge u+d^2\theta\wedge w, \\
u\wedge{\sf d}^2\theta &=& {\sf d}^2\theta\wedge u + \left[(j-j^2) \, x^{-1}\theta - j^2h\right] {\sf d}^2x\wedge u.
\end{eqnarray*}

\item Relations between the Cartan-Maurer forms
\begin{eqnarray*}
w\wedge u&=& u\wedge w, \\
w\wedge w\wedge w &=& 0.
\end{eqnarray*}
\end{itemize}

\section{A Z$_3$-graded $(h,j)$-deformed quantum (super)group}

In this section, we will consider the $Z_3$-graded structures of the $(h,j)$-deformed quantum 2x2-supermatrices. We had given commutation relations between coordinates in the $h$-superplane in (\ref{eq8}). Here the coordinate $x$ with the respect to the $Z_3$-grading is of grade 0 and the coordinate $\theta$ with respect to the $Z_3$-grading is of grade 1.

The noncommutative space ${\mathbb R}_h(1|1)$ with the function algebra
\[O({\mathbb R}_h(1|1))=K\{x,\theta\}/(x\theta-\theta x-hx^2, \quad \theta^3, \quad h^3)\]
is called Z$_3$-graded $h$-superplane. The noncommutative space ${\mathbb R}_{h,j}^*(1|1)$ with the function algebra
\[O({\mathbb R}_{h,j}^*(1|1))=K\{\varphi,y\}/(\varphi y-jy\varphi-hj^2\varphi^2, \quad \varphi^3, \quad h^3)\]
is called dual Z$_3$-graded $h$-superplane.

Under these definitions, we have
\begin{equation} \label{eq43}
{\mathcal R}_h(1|1)=\left\{\begin{pmatrix} x \\ \theta \end{pmatrix}: x\theta=\theta x+hx^2, \quad \theta^3=0, \quad h^3=0.\right\}
\end{equation}
and
\begin{equation} \label{eq44}
{\mathcal R}^*_{q,j}(1|1) \left\{\begin{pmatrix} \varphi \\ y \end{pmatrix}: \varphi y=jy\varphi+hj^2\varphi^2, \quad \varphi^3=0, \quad h^3=0.\right\}
\end{equation}
Here,
$$ \left[{\mathcal R}_h(1|1)\right]^* = {\mathcal R}^*_{h,j}(1|1).$$

Let $T$ be a $2x2$ supermatrix in $Z_3$-graded superspace
\begin{equation}  \label{eq45}
T=\begin{pmatrix} a \quad \beta \\ \gamma \quad d \end{pmatrix}
\end{equation}
where $a$ and $d$ with respect to the $Z_3$-grading is of grade $0$, and $\beta$ and $\gamma$  with the respect to the $Z_3$-grading is of grade $2$ and grade $1$, respectively. We now consider linear transformations with the following properties:
\begin{equation}  \label{eq46}
T: {\mathcal R}_h(1|1)\longrightarrow{\mathcal R}_h(1|1), \quad T:{\mathcal R}^*_{h,j}(1|1)\longrightarrow{\mathcal R}^*_{h,j}(1|1).
\end{equation}
We assume that the entries of $T$ are $j$-commutative with the elements of ${\mathcal R}_h(1|1)$ and ${\mathcal R}^*_{h,j}(1|1)$, i.e. for example,
\[ax=xa, \quad \theta\beta=j^2\beta\theta, \, etc.\]
As a consequence of linear transformations in (\ref{eq46}), the elements
\begin{equation} \label{eq47}
\tilde{x}=ax+\beta\theta, \quad \tilde{\theta}=\gamma x+d\theta
\end{equation}
should satisfy the relations in (\ref{eq43}):
\[\tilde{x}\tilde{\theta}=\tilde{\theta}\tilde{x}+h\tilde{x}^2, \quad \tilde{\theta}^3=0. \]
Using these relations one has,
\[a\gamma=\gamma a+h[a^2-ad+\gamma\beta+j^2ha\beta], \quad d\gamma=\gamma d,\]
\[\beta d=j^2[d\beta+h\beta^2], \quad \gamma^3=-hj\left[(j-1)\gamma^2d+2jh\gamma d^2\right].\]
Similarly, the elements
\begin{equation} \label{eq48}
\tilde{\varphi}=a\varphi+j^2\beta y, \quad \tilde{y}=jy\varphi+dy
\end{equation}
must satisfy the relations in (\ref{eq44}). Using these relations, one has
\[a\beta=j\beta a, \quad \beta^3=0.\]
Also if we use the second relation in (\ref{eq19}),
\[\tilde{x}\tilde{y}=\tilde{y}\tilde{x}+(j^2-1)\tilde{\varphi}\tilde{\theta}+hj\tilde{\varphi}\tilde{x},\]
we have,
\[ad=da+(1-j)\beta\gamma+h\beta a, \quad \beta \gamma= \gamma\beta+ha\beta. \]

Consequently we have the following commutation relations between the matrix elements of $T$:
\begin{eqnarray} \label{eq49}
a\beta &= & j \,\beta a, \nonumber\\
a\gamma &=& \gamma a+h\left[a^2-ad+\gamma\beta+j^2ha\beta\right], \nonumber\\
d\beta &= & j \, \beta d + jh \, \beta^2, \nonumber\\
d\gamma &=& \gamma d, \nonumber\\
\beta^3 &=& 0,\qquad \gamma^3 = -hj\left[(j-1)\gamma^2d + 2jh \,\gamma d^2\right], \nonumber\\
\beta \gamma &=& \gamma\beta + h \, a\beta, \nonumber\\
ad &=& da + (1-j) \, \beta\gamma +h \, \beta a.
\end{eqnarray}

Super inverse and super determinant of  of $T$ is defined in \cite{5}, as follows.
\begin{equation} \label{eq50}
T^{-1}=\begin{pmatrix} A & -a^{-1}\beta d^{-1}-a^{-1}\beta d^{-1}\gamma a^{-1}\beta d^{-1} \cr
-d^{-1}\gamma a^{-1}-d^{-1}\gamma a^{-1}\beta d^{-1}\gamma a^{-1} & D \end{pmatrix}\end{equation}
where
\[A= a^{-1}+a^{-1}\beta d^{-1}\gamma a^{-1}+a^{-1}\beta d^{-1}\gamma a^{-1} \beta d^{-1}\gamma a^{-1}, \]
\[D=d^{-1}+d^{-1}\gamma a^{-1}\beta d^{-1}+d^{-1}
\gamma a^{-1}\beta d^{-1}\gamma a^{-1}\beta d^{-1}\]
and
\begin{equation} \label{eq51}
D_{h,j}(T)=ad^{-1}+ad^{-1}\gamma a^{-1}\beta d^{-1}+ad^{-1}\gamma a^{-1}\beta d^{-1}\gamma a^{-1}\beta d^{-1}.
\end{equation}

\begin{defn} Z$_3$-graded $(h,j)$-deformed supergroup is a group that consists $T$ matrices that satisfy the following three conditions.
\begin{itemize}
\item Elements of a matrix $T$ satisfies relations given with (\ref{eq49}),
\item $T$ matrix has the inverse given with (\ref{eq50}),
\item $T$ matrix has the super determinant given with (\ref{eq51}).
\end{itemize}
\end{defn}

This defined group, will be denoted with GL$_{h,j}(1|1)$. It can be shown that the Z$_3$-graded quantum supergroup GL$_{h,j}(1|1)$ is a Z$_3$-graded Hopf (super) algebra. A study on this group and on differential geometry of this group is in progress.

\section*{Acknowledgments}

This work was supported in part by {\bf TBTAK} the Turkish Scientific and Technical Research Council.

\end{document}